\documentclass[10pt]{amsart}
\usepackage{amsfonts}
\usepackage{amsthm}
\usepackage{amsmath}
\usepackage{graphicx}
\usepackage{amscd,amssymb,amsthm}

\newtheorem{theorem}{Theorem}
\newtheorem{lemma}{Lemma}

\newcommand{\real}{\operatorname{Re}}
\newcommand{\imag}{\operatorname{Im}}

\begin{document}

\title[Tur\'an type inequalities for Lommel functions]{Tur\'an type inequalities for some Lommel functions of the first kind}

\author{\'Arp\'ad Baricz}
\address{Department of Economics, Babe\c{s}-Bolyai University, Cluj-Napoca 400591, Romania} \email{bariczocsi@yahoo.com}

\author{Stamatis Koumandos}
\address{Department of Mathematics and  Statistics, The University of Cyprus, P.O. Box 20537, 1678 Nicosia, Cyprus} \email{skoumand@ucy.ac.cy}

\keywords{Lommel functions of the first kind, zeros of Lommel functions of the first kind, Tur\'an type inequalities, monotone form of l'Hospital's rule, infinite product representation, Laguerre-P\'olya class}

\subjclass[2010]{33C10, 33B10, 42A05.}

\maketitle

\begin{abstract}
 In this paper certain Tur\'{a}n type inequalities for some Lommel functions of the first kind are deduced. The key tools in our proofs are the infinite product representation for these Lommel functions of the first kind, a classical result of G. P\'olya on the zeros of some particular entire functions, and the connection of these Lommel functions with the so-called Laguerre-P\'olya class of entire functions. Moreover, it is shown that in some cases J. Steinig's results on the sign of Lommel functions of the first kind combined with the so-called monotone form of l'Hospital's rule can be used in the proof of the corresponding Tur\'an type inequalities.
\end{abstract}

\section{Introduction}

The Tur\'an type inequalities for (orthogonal) polynomials and special functions have attracted many mathematicians starting from 1948, when G. Szeg\H o \cite{szego} published four different proofs of P. Tur\'an's famous inequality on Legendre polynomials \cite{turan}. In the last 65 years it was shown by several researchers that the most important (orthogonal) polynomials and special functions satisfy some Tur\'an type inequalities. Recently, these kind of inequalities have attracted again the attention of many researchers because some of the Tur\'an type inequalities have been applied in different problems. For more details the interested reader is referred to some very recent papers on the subject \cite{b1,b2,b3,kp} and to the references therein. In this paper we make a contribution to the subject by proving the corresponding Tur\'an type inequalities for a particular Lommel function of the first kind. These Lommel functions of the first kind are important because arise in the theory of positive trigonometric sums, see \cite{koumlam} for more details.

The Lommel function of the first kind $s_{\mu,\nu}$ is a particular solution of the inhomogeneous Bessel differential equation
$$z^{2}y''(z) + zy'(z) + (z^{2}-\nu^{2})y(z) = z^{\mu+1},$$
and it can be expressed in terms of a hypergeometric series
\begin{equation}\label{hyp}
s_{\mu,\nu}(z)=\frac{z^{\mu+1}}{(\mu-\nu+1)(\mu+\nu+1)}{}_{1}F_{2}\left(1;\frac{\mu-\nu+3}{2},\frac{\mu+\nu+3}{2};-\frac{z^2}{4}\right).
\end{equation}
We note that for $\mu$, $\nu\in\mathbb{C}$ with $\real(\mu\pm\nu+1)>0$ and $z\in\mathbb{C}\setminus(-\infty, 0]$ we have the following integral representation
\begin{equation}\label{int}
  s_{\mu,\nu}(z) = \frac{\pi}{2} \left[Y_{\nu}(z) \int_{0}^{z} t^{\mu}
    J_{\nu}(t)\, dt - J_{\nu}(z) \int_{0}^{z}t^{\mu} Y_{\nu}(t)\,dt\right],
\end{equation}
where $J_{\nu}$ and $Y_{\nu}$ are the usual Bessel functions of the first and second kind. It is also important to mention here that in 1972, J. Steinig \cite[Theorem 2]{stein} examined the sign of $s_{\mu,\nu}(z)$ for real $\mu,\nu$ and positive $z$. He showed, among other things, that for $\mu<\frac{1}{2}$ the function  $s_{\mu,\nu}$ has infinitely many changes of sign on $(0,\infty)$. See also \cite{ga} for related considerations. In \cite{koumlam} estimates for the location of the zeros of the function $s_{\mu-\frac{1}{2},\frac{1}{2}}$, where $\mu\in(0,1)$, have been obtained. Various properties of the zeros of $s_{\mu-k-\frac{1}{2},\frac{1}{2}}$, where $k\in\{0,1,\ldots\}$ and $\mu\in(0,1)$, have been established in \cite{koum}. We note that although the Lommel functions of the first kind occur in several places in physics and engineering, relatively very little has been done in the literature for the Lommel functions of the first kind.

In this paper, our aim is to prove the following main result:

\begin{theorem}\label{Th1}
If $z>0$ and $\mu\in\left(-\frac{5}{2},-\frac{1}{2}\right),$ $\mu\neq-\frac{3}{2},$  then the following Tur\'{a}n type inequality is valid
\begin{equation}\label{turan}
\left[s_{\mu,\frac{1}{2}}(z)\right]^2-s_{\mu-1,\frac{1}{2}}(z)s_{\mu+1,\frac{1}{2}}(z)>\frac{1}{\frac{1}{2}-\mu}\left[s_{\mu,\frac{1}{2}}(z)\right]^2.
\end{equation}
\end{theorem}

\section{Lemmas}
For the proof of Theorem \ref {Th1} we need the following lemmas.

\begin{lemma}\label{lem1}
Let
$$\varphi_{k}(z)={}_{1}F_{2}\left(1;\,\frac{\mu-k+2}{2},\frac{\mu-k+3}{2};-\frac{z^2}{4}\right),$$
where $z\in\mathbb{C},$ $\mu\in\mathbb{R}$ and $k\in\{0,1,\ldots\}$ such that $\mu-k$ is not in $\{0,-1,\dots\}.$ Then, $\varphi_{k}$ is an even real entire function of order $\rho=1$ and of exponential type $\tau=1$. Moreover, $\varphi_{k}$ is of genus 1. The Hadamard's factorization of $\varphi_{k}$ is of the form
$$
\varphi_{k}(z)=\prod_{n\geq1}\left(1-\frac{z^2}{z_{\mu,k,n}^2}\right),
$$
where $\pm z_{\mu,k,1}, \pm z_{\mu,k,2},\ldots$ are all zeros of the function $\varphi_{k}$ and the infinite product is absolutely convergent.
\end{lemma}
\begin{proof}
The Pochhammer symbol $(a)_{n}$ is defined by
$$
(a)_{0}=1, \quad
(a)_{n}=a(a+1)\ldots(a+n-1)=\frac{\Gamma(n+a)}{\Gamma(a)},\ \ \ n\in\{1,2,\dots\}.
$$
Using the duplication formula \cite[p. 22]{aar}
$$
(2a)_{2n}=(a)_{n}\left(a+\frac{1}{2}\right)_{n}2^{2n},
$$
we find that for $k\in\{0,1,\ldots\}$
$$
\varphi_{k}(z)=\sum_{n\geq0}\frac{(-1)^{n}\,z^{2n}}{(\mu-k+2)_{2n}}.
$$
Taking into consideration the well-known limits
$$
\lim_{n\to \infty} \frac{\log \Gamma(n+c)}{n\,\log n}=1,\quad
\lim_{n\to \infty} \frac{[\Gamma(n+c)]^{1/n}}{n}=\frac{1}{e},
$$
where $c$ is a positive constant, and \cite[p. 6, Theorems 2 and 3]{lev}, we infer that the entire function $\varphi_{k}$ is of order $\rho=1$ and of exponential type $\tau=1$. The rest of the assertions of the Lemma follow by applying Hadamard's Theorem \cite[p. 26]{lev}
and Lindel\"{o}f's Theorem \cite[p. 33]{lev}.
\end{proof}

\begin{lemma}\label{lem2}
For $z,$ $\mu$ and $k$ as in Lemma \ref{lem1} we have
\begin{equation}\label{rec}
(\mu-k+1)\varphi_{k+1}(z)=(\mu-k+1)\varphi_{k}(z)+z\varphi_{k}'(z).
\end{equation}
\end{lemma}

\begin{proof}
It follows from \eqref{hyp} that
\begin{equation}\label{id}
\sqrt{z}s_{\mu-k-\frac{1}{2},\frac{1}{2}}(z)=\frac{z^{\mu-k+1}}{(\mu-k)\,(\mu-k+1)}\varphi_{k}(z).
\end{equation}
Differentiating both sides of the relation \eqref{id} and by using the known formulae \cite[p. 348]{wats}
\begin{align}\label{diff}
&\left[z^{\nu}s_{\mu,\nu}(z)\right]'=(\mu+\nu-1)z^{\nu}s_{\mu-1,\nu-1}(z),\\
\nonumber &s_{\mu,-\nu}(z)=s_{\mu,\nu}(z),
\end{align}
we obtain the recurrence relation \eqref{rec}. The condition that $\mu\in\mathbb{R}$ and
$k\in\{0,1,\ldots\}$ are such that $\mu-k$ is not in $\{0,-1,\dots\}$ is required since the Lommel function $s_{\mu,\nu}$ is undefined when either of the numbers $\mu\pm\nu$ is an odd negative integer, according to \cite[p. 345]{wats}.
\end{proof}

\begin{lemma}\label{lem3}
For $\mu>0$ we have
$$z\varphi_{0}(z)=\mu(\mu+1)\int_{0}^{1}(1-t)^{\mu-1}\sin(zt)dt,$$
$$\varphi_{1}(z)=\mu\int_{0}^{1}(1-t)^{\mu-1}\cos(zt)dt.$$
\end{lemma}

\begin{proof}
Recall that
\begin{equation*}
  J_{\frac{1}{2}}(z) = \sqrt{\frac{2}{\pi z}} \sin z \quad\mbox{and}\quad
  Y_{\frac{1}{2}}(z) = -\sqrt{\frac{2}{\pi z}} \cos z\,.
\end{equation*}
By using \eqref{int} and  \eqref{id} we find that
\begin{align}\label{int1}
\frac{z^{\mu+1}}{\mu(\mu+1)}\varphi_{0}(z)&=\sqrt{z}s_{\mu-\frac{1}{2},\frac{1}{2}}(z) =\int_{0}^{z} t^{\mu-1} \sin(z-t)dt\\
\nonumber &=z^{\mu}\int_{0}^{1} (1-t)^{\mu-1} \sin(zt)dt.
\end{align}
By  \eqref{id}, \eqref{diff} and the above relation we obtain
\begin{align}\label{int2}
\frac{z^{\mu}}{\mu}\varphi_{1}(z)&=(\mu-1)\sqrt{z}s_{\mu-\frac{3}{2},\frac{1}{2}}(z) =\int_{0}^{z} t^{\mu-1} \cos(z-t)dt \\
\nonumber &=z^{\mu}\int_{0}^{1} (1-t)^{\mu-1} \cos(zt)dt.
\end{align}
From \eqref{int1} and \eqref{int2} the assertions of the Lemma follow.
\end{proof}

We also need the following result, which corresponds to a famous theorem of G. P\'{o}lya \cite{pol}. See also \cite{ki-kim} and \cite{sz} for different proofs of this result.
\begin{lemma}\label{pol}
Suppose that the function $f$ is positive, strictly increasing and continuous on $[0, 1)$ and that
$\int_{0}^{1}f(t)dt<\infty$. Then, the entire functions
$$u(z)=\int_{0}^{1}f(t)\sin(zt)dt\ \ \ \mbox{and} \ \ \ v(z)=\int_{0}^{1}f(t)\cos(zt)dt$$
have only real and simple zeros and their zeros interlace.
\end{lemma}

\section{Proof of Theorem \ref{Th1}}

\begin{proof}
Observe that by Lemmas \ref{lem3} and \ref{pol} we have that for $\mu\in(0,1)$ the function $\varphi_{0}$ has only real and simple zeros.
For $n\in\{1,2,\dots\}$ let $\xi_{\mu,n}:=z_{\mu,0,n}$ be the $n$th positive zero of $\varphi_{0},$ and let $\xi_{\mu,0}=0$.
Lemma \ref{lem1} yields
$$
\varphi_{0}(z)=\prod_{n\geq1}\left(1-\frac{z^2}{\xi_{\mu,n}^2}\right),
$$
and consequently
\begin{equation}\label{eq1}
\frac{\varphi_{0}'(z)}{\varphi_{0}(z)}=\sum_{n\geq1}\frac{2z}{z^2-\xi_{\mu,n}^2}.
\end{equation}
From \eqref{rec} we obtain
\begin{equation}\label{eq2}
(\mu+1)\varphi_{1}(z)=(\mu+1)\varphi_{0}(z)+z\varphi_{0}'(z).
\end{equation}
Since the zeros of the function $\varphi_{0}$ are simple, equation \eqref{eq2} implies that $\varphi_{0}$ and $\varphi_{1}$ have no common zeros. Note that
 $\varphi_{0}(0)=\varphi_{1}(0)=1$. Combining \eqref{eq1} and \eqref{eq2} we have
$$\frac{\varphi_{1}(z)}{z\varphi_{0}(z)}=\frac{1}{z}+\frac{1}{\mu+1}\sum_{n\geq1}\frac{2z}{z^2-\xi_{\mu,n}^2}$$
for $z\neq 0,$ $z\neq \pm \xi_{\mu,n},$ $n\in\{1,2,\dots\}$.
From the above Mittag-Leffler expansion for all $z\in(\xi_{\mu,n-1},\xi_{\mu,n}),$ $n\in\{1,2\dots\}$ we get
$$\left[\frac{\varphi_{1}(z)}{z\varphi_{0}(z)}\right]'<0.$$
Now, since $\varphi_{0}'(\xi_{\mu,n})\neq 0$ and $\varphi_{1}(\xi_{\mu,n})\neq 0$ for all $n\in\{1,2,\dots\},$ we deduce from the above inequality that for all $z>0$ we have
\begin{equation}\label{eq5}
z\varphi_{0}(z)\varphi_{1}'(z)-\varphi_{0}(z)\varphi_{1}(z)-z\varphi_{1}(z)\varphi_{0}'(z)<0.
\end{equation}
On the other hand, in view of \eqref{id} we have
\begin{equation*}
(\mu+1)\frac{\varphi_{1}(z)}{z\,\varphi_{0}(z)}=(\mu-1)\frac{s_{\mu-\frac{3}{2},\frac{1}{2}}(z)}{s_{\mu-\frac{1}{2},\frac{1}{2}}(z)}.
\end{equation*}
Using the differentiation formula \eqref{diff} and the above relation, we see that, for $\mu\in(0,1)$ and $z>0$, the inequality \eqref{eq5} is equivalent to
\begin{equation}\label{ineqvarphi0}
(\mu-2)s_{\mu-\frac{5}{2},\frac{1}{2}}(z)s_{\mu-\frac{1}{2},\frac{1}{2}}(z)-(\mu-1)\left[s_{\mu-\frac{3}{2},\frac{1}{2}}(z)\right]^2>0,
\end{equation}
and changing in this inequality $\mu$ to $\mu+\frac{3}{2},$ the desired inequality \eqref{turan} follows for $\mu\in\left(-\frac{3}{2},-\frac{1}{2}\right).$

Lemma \ref{pol} implies also that for $\mu\in(0,1)$ the function $\varphi_{1}$ has only real and simple zeros. Then applying \eqref{rec} for $k=1$, we conclude that the functions $\varphi_{1}$ and $\varphi_{2}$ have no common zeros. Let $\zeta_{\mu,n}:=z_{\mu,1,n}$ be the $n$th positive zero of $\varphi_{1}$.
As above, by Lemma \ref{lem1} and \eqref{rec} we have the next Mittag-Leffler expansion
\begin{equation}\label{eq6}
\frac{\varphi_{2}(z)}{z\varphi_{1}(z)}=\frac{1}{z}+\frac{1}{\mu}\sum_{n\geq1}\frac{2z}{z^2-\zeta_{\mu,n}^2}.
\end{equation}
Hence, if $z=x+\mathrm{i}y$,
\begin{equation*}
\imag\left[\frac{\varphi_{2}(z)}{z\varphi_{1}(z)}\right]=-\frac{y}{\mu}\left\{\frac{\mu}{x^2+y^2}+
\sum_{n\geq1}\left[\frac{1}{(x+\zeta_{\mu,n})^2+y^2}+\frac{1}{(x-\zeta_{\mu,n})^2+y^2}\right]\right\},
\end{equation*}
which is zero only if $y=0$. Therefore the function $\varphi_{2}$ has only real zeros and from \eqref{eq6} for $z>0$ we derive the inequality
\begin{equation}\label{eq7}
z\varphi_{1}(z)\varphi_{2}'(z)-\varphi_{1}(z)\varphi_{2}(z)-z\varphi_{2}(z)\varphi_{1}'(z)<0.
\end{equation}
This implies also that all zeros of $\varphi_{2}$ are simple. On the other hand, from \eqref{id} we have
\begin{equation*}
\mu\frac{\varphi_{2}(z)}{z\varphi_{1}(z)}=(\mu-2)\frac{s_{\mu-\frac{5}{2},\frac{1}{2}}(z)}{s_{\mu-\frac{3}{2},\frac{1}{2}}(z)}.
\end{equation*}
From this and \eqref{diff} we conclude that for $\mu\in(0,1)$ and $z>0$ inequality \eqref{eq7} is equivalent to
\begin{equation}\label{ineqvarphi1}
(\mu-3)s_{\mu-\frac{7}{2},\frac{1}{2}}(z)s_{\mu-\frac{3}{2},\frac{1}{2}}(z)-(\mu-2)\left[s_{\mu-\frac{5}{2},\frac{1}{2}}(z)\right]^2>0,
\end{equation}
and changing in this inequality $\mu$ to $\mu+\frac{5}{2},$ the desired inequality \eqref{turan} follows for $\mu\in\left(-\frac{5}{2},-\frac{3}{2}\right).$
\end{proof}

\section{Concluding remarks}

In this section our aim is to comment and complement the proof of the main result.

{\bf A.} First, we would like to present a somewhat alternative proof of the Tur\'an type inequality \eqref{turan}. For this, let us recall that by definition the real entire function $\phi,$
defined by
\begin{equation}\label{entire}\phi(z)=\varphi(z;t)=\sum_{n\geq0}b_n(t)\frac{z^n}{n!},\end{equation} is said
to be in the Laguerre-P\'olya class (denoted by $\mathcal{LP}$), if
$\phi(z)$ can be expressed in the form
$$\phi(z)=cz^de^{-\alpha z^2+\beta z}\prod_{n=1}^{\omega}\left(1-\frac{z}{z_n}\right)e^{\frac{z}{z_n}},\ \ \ 0\leq\omega\leq\infty,$$
where $c$ and $\beta$ are real, $z_n$'s are real and nonzero for all
$n\in\{1,2,{\dots},\omega\},$ $\alpha\geq0,$ $d$ is a nonnegative
integer and $\sum_{n=1}^{\omega}z_i^{-2}<\infty.$ If $\omega=0,$
then, by convention, the product is defined to be $1.$ For the various properties of the functions in the Laguerre-P\'olya class we
refer to \cite{craven1,craven2,norfolk,csordas} and to the
references therein. We note that in fact a real entire
function $\phi$ is in the Laguerre-P\'olya class if and only if
$\phi$ can be uniformly approximated on disks around the origin
by a sequence of polynomials with only real zeros. This implies that the class $\mathcal{LP}$ is closed under
differentiation, that is, if $\phi\in\mathcal{LP}$, then
$\phi^{(m)}\in\mathcal{LP}$ for all $m$ nonnegative integer. We also recall the following result (for more details we refer to H. Skovgaard's paper \cite{skov}): if a real entire function $\phi$ belongs to the Laguerre-P\'olya class $\mathcal{LP}$ then satisfies the Laguerre
type inequalities
\begin{equation}\label{laguerre}\left[\phi^{(m)}(z)\right]^2-\phi^{(m-1)}(z)\phi^{(m+1)}(z)\geq0,\end{equation}
for $m\in\{1,2,\dots\}.$ Now, recall that the zeros of $\varphi_0$ are real and satisfy the inequalities (see \cite{koumlam}) $\xi_{\mu,2n+1}>\xi_{\mu,2n}>2n\pi,$ where $\mu\in(0,1)$ and $n\in\{1,2,\dots\}.$ Combining this with the fact that the series $\sum_{n\geq 1}n^{-2}$ converges, the comparison test yields that $\sum_{n\geq 1}\xi_{\mu,n}^{-2}$ converges. Note that this result can be verified also by means of Lemma \ref{lem1}, since the function $\varphi_0$ has genus $1,$ order $1$ and therefore the convergence exponent of their zeros is also $1.$ Now, the convergence of the above series together with the infinite product representation of the function $\varphi_0$ yields that $\varphi_0\in\mathcal{LP},$ which in turn implies that the above function satisfies the Laguerre inequality \eqref{laguerre}. In particular, by choosing $m=1,$ for $\mu\in(0,1)$ and $z>0$ we have that
$$\left[\varphi_0'(z)\right]^2-\varphi_0(z)\varphi_0''(z)>0.$$
Combining this with \eqref{eq2} and $$\varphi_1'(z)=\frac{\mu+2}{\mu+1}\varphi_0'(z)+\frac{z}{\mu+1}\varphi_0''(z),$$
we obtain that
\begin{align*}z&\varphi_{0}(z)\varphi_{1}'(z)-\varphi_{0}(z)\varphi_{1}(z)-z\varphi_{1}(z)\varphi_{0}'(z)\\&=
-\left[\varphi_0(z)\right]^2+\frac{z^2}{\mu+1}\left[\varphi_0(z)\varphi_0''(z)-\left[\varphi_0'(z)\right]^2\right]<0,\end{align*}
that is, the inequality \eqref{eq5} is valid for $\mu\in(0,1)$ and $z>0.$ The proof of inequality \eqref{eq7} can be done in a similar way. More precisely, according to Lemma \eqref{pol} the zeros of $\varphi_0$ and $\varphi_1$ interlace. This means that $\sum_{n\geq 1}\zeta_{\mu,n}^{-2}$ also converges, and combining this with the infinite product representation of the function $\varphi_1$ yields that $\varphi_1\in\mathcal{LP}.$ We note that the convergence of $\sum_{n\geq 1}\zeta_{\mu,n}^{-2}$ can be verified also by means of Lemma \ref{lem1}, since the function $\varphi_1$ has genus $1,$ order $1$ and therefore the convergence exponent of their zeros is also $1.$ Now, this in turn implies that the above function satisfies the Laguerre inequality \eqref{laguerre}. In particular, by choosing $m=1,$ for $\mu\in(0,1)$ and $z>0$ we have that
$$\left[\varphi_1'(z)\right]^2-\varphi_1(z)\varphi_1''(z)>0.$$
Combining this with $$\varphi_2(z)=\varphi_1(z)+\frac{z}{\mu}\varphi_1'(z),$$
$$\varphi_2'(z)=\frac{\mu+1}{\mu}\varphi_1'(z)+\frac{z}{\mu}\varphi_1''(z),$$
we obtain that
\begin{align*}z&\varphi_{1}(z)\varphi_{2}'(z)-\varphi_{1}(z)\varphi_{2}(z)-z\varphi_{2}(z)\varphi_{1}'(z)\\&=
-\left[\varphi_1(z)\right]^2+\frac{z^2}{\mu}\left[\varphi_1(z)\varphi_1''(z)-\left[\varphi_1'(z)\right]^2\right]<0,\end{align*}
that is, the inequality \eqref{eq7} is valid for $\mu\in(0,1)$ and $z>0.$

{\bf B.} Now, we would like to present the following version of the monotone form of l'Hospital rule due to I.
Pinelis \cite{pin}. We note that another version of monotone form of l'Hospital rule was
proved by G.D. Anderson, M.K. Vamanamurthy and M. Vuorinen \cite{anderson2,anderson}.

\begin{lemma}\label{lopi}
Let $-\infty\leq a<b\leq\infty$ and let $f$ and $g$ be
differentiable functions on $(a,b).$ Assume that either $g'>0$
everywhere on $(a,b)$ or $g'<0$ on $(a,b).$ Furthermore, suppose
that $f(a^{+})=g(a^{+})=0$ or $f(b^{-})=g(b^{-})=0$ and $f'/g'$ is
(strictly) increasing (decreasing) on $(a,b).$ Then the ratio $f/g$
is (strictly) increasing (decreasing) too on $(a,b).$
\end{lemma}

In this paper we proved that for $z>0$ and $\mu\in\left(-\frac{5}{2},-\frac{1}{2}\right),$ $\mu\neq-\frac{3}{2},$ the Tur\'an type inequality \eqref{turan} is valid. It is natural to ask whether for other values of $\mu$ the inequality \eqref{turan} is valid or not. Based on numerical experiments we believe but ar unable to prove the following conjecture.

\vskip2mm

{\bf Conjecture}\ The Tur\'an type inequality \eqref{turan} is valid for $z>0$ and $\mu\geq\frac{3}{2}$ and fails to hold for $z>0$ and $\mu\in\left(-\frac{1}{2},\frac{3}{2}\right).$

\vskip2mm

 Below we show that to prove the first part of the above conjecture in fact it is enough to show that \eqref{turan} is valid for $z>0$ and $\mu\in\left[\frac{3}{2},\frac{5}{2}\right).$ More precisely, we show that when $\mu\geq\frac{3}{2}$ the inequality \eqref{turan} can be proved iteratively in the sense that the following chain of implications is valid: if \eqref{turan} is valid for $\mu\in\left[\frac{3}{2},\frac{5}{2}\right),$ then this implies that the Tur\'an type inequality \eqref{turan} is also valid for $\mu\in\left[\frac{5}{2},\frac{7}{2}\right),$ and this implies that \eqref{turan} holds true for $\mu\in\left[\frac{7}{2},\frac{9}{2}\right),$ and so on. For this first recall that according to J. Steinig \cite[Theorem 3]{stein} we have that $s_{\mu,\nu}(z)>0$ for $z>0$ and $\mu>\frac{1}{2}.$ On the other hand, by using the recurrence relation \cite[p. 348]{wats}
$$s_{\mu,\nu}'(z)+\frac{\nu}{z}s_{\mu,\nu}(z)=(\mu+\nu-1)s_{\mu-1,\nu-1}(z),$$ we get
    $$s_{\mu,\frac{1}{2}}'(z)+\frac{1}{2z}s_{\mu,\frac{1}{2}}(z)=\left(\mu-\frac{1}{2}\right)s_{\mu-1,\frac{1}{2}}(z),$$
    $$s_{\mu+1,\frac{1}{2}}'(z)+\frac{1}{2z}s_{\mu+1,\frac{1}{2}}(z)=\left(\mu+\frac{1}{2}\right)s_{\mu,\frac{1}{2}}(z),$$
    and consequently
    \begin{align*}&s_{\mu+1,\frac{1}{2}}'(z)s_{\mu,\frac{1}{2}}(z)-s_{\mu+1,\frac{1}{2}}(z)s_{\mu,\frac{1}{2}}'(z)\\& \ \ =
    \left(\mu+\frac{1}{2}\right)\left[s_{\nu,\frac{1}{2}}(z)\right]^2-\left(\mu-\frac{1}{2}\right)s_{\nu-1,\frac{1}{2}}(z)s_{\mu+1,\frac{1}{2}}(z).\end{align*}
    Now, suppose that the Tur\'an type inequality \eqref{turan} is valid for $\mu\in\left[\frac{3}{2},\frac{5}{2}\right)$ and $z>0.$ The above relation shows that \eqref{turan} is equivalent to the fact that the function $z\mapsto s_{\mu+1,\frac{1}{2}}(z)/s_{\mu,\frac{1}{2}}(z)$ is increasing on $(0,\infty)$ for $\mu\in\left[\frac{3}{2},\frac{5}{2}\right).$ Changing $\mu$ to $\mu-1,$ we get that the function $z\mapsto s_{\mu,\frac{1}{2}}(z)/s_{\mu-1,\frac{1}{2}}(z)$ is increasing on $(0,\infty)$ for $\mu\in\left[\frac{5}{2},\frac{7}{2}\right).$ But, this means that the function
    $$z\mapsto\frac{\left[\sqrt{z}s_{\mu+1,\frac{1}{2}}(z)\right]'}{\left[\sqrt{z}s_{\mu,\frac{1}{2}}(z)\right]'}=
    \frac{\left(\mu+\frac{1}{2}\right)\sqrt{z}s_{\mu,\frac{1}{2}}(z)}{\left(\mu-\frac{1}{2}\right)\sqrt{z}s_{\mu-1,\frac{1}{2}}(z)}=
    \frac{\mu+\frac{1}{2}}{\mu-\frac{1}{2}}\frac{s_{\mu,\frac{1}{2}}(z)}{s_{\mu-1,\frac{1}{2}}(z)}$$
    is also increasing on $(0,\infty)$ for $\mu\in\left[\frac{5}{2},\frac{7}{2}\right).$ Now, since $\sqrt{z}s_{\mu+1,\frac{1}{2}}(z)\to0$ and $\sqrt{z}s_{\mu,\frac{1}{2}}(z)\to0$ as $z\to0,$ by using the monotone form of l'Hospital's rule (see Lemma \ref{lopi}) we obtain that the function
    $$z\mapsto\frac{\sqrt{z}s_{\mu+1,\frac{1}{2}}(z)}{\sqrt{z}s_{\mu,\frac{1}{2}}(z)}=\frac{s_{\mu+1,\frac{1}{2}}(z)}{s_{\mu,\frac{1}{2}}(z)}$$
    is also increasing on $(0,\infty)$ for $\mu\in\left[\frac{5}{2},\frac{7}{2}\right),$ and this is equivalent to inequality \eqref{turan}. A similar procedure shows that the fact that \eqref{turan} is valid for $\mu\in\left[\frac{5}{2},\frac{7}{2}\right)$ implies that it is also valid for $\mu\in\left[\frac{7}{2},\frac{9}{2}\right),$ and so on.

{\bf C.} We note that by using the above argument it can be shown that for $z\in(0,\xi_{\mu-1,1})$ and $\mu\in\left(-\frac{1}{2},\frac{1}{2}\right)$ the Tur\'an type inequality \eqref{turan} is reversed, which is in agreement with the second part of the above conjecture. Namely, taking into account the fact that \eqref{turan} is valid for $\mu\in\left(-\frac{3}{2},-\frac{1}{2}\right)$ and $z>0,$ we get that the function $z\mapsto s_{\mu+1,\frac{1}{2}}(z)/s_{\mu,\frac{1}{2}}(z)$ is decreasing on $(\xi_{\mu,n-1},\xi_{\mu,n}),$ $n\in\{1,2,\dots\},$ for $\mu\in\left(-\frac{3}{2},-\frac{1}{2}\right).$ Now, changing $\mu$ to $\mu-1$ we get that the function $z\mapsto s_{\mu,\frac{1}{2}}(z)/s_{\mu-1,\frac{1}{2}}(z)$ is also decreasing on $(0,\xi_{\mu-1,1})$  for $\mu\in\left(-\frac{1}{2},\frac{1}{2}\right).$ But, this means that the function
    $$z\mapsto\frac{\left[\sqrt{z}s_{\mu+1,\frac{1}{2}}(z)\right]'}{\left[\sqrt{z}s_{\mu,\frac{1}{2}}(z)\right]'}=
    \frac{\left(\mu+\frac{1}{2}\right)\sqrt{z}s_{\mu,\frac{1}{2}}(z)}{\left(\mu-\frac{1}{2}\right)\sqrt{z}s_{\mu-1,\frac{1}{2}}(z)}=
    \frac{\mu+\frac{1}{2}}{\mu-\frac{1}{2}}\frac{s_{\mu,\frac{1}{2}}(z)}{s_{\mu-1,\frac{1}{2}}(z)}$$
    is increasing on $(0,\xi_{\mu-1,1})$ for $\mu\in\left(-\frac{1}{2},\frac{1}{2}\right).$ Now, since $\sqrt{z}s_{\mu+1,\frac{1}{2}}(z)\to0$ and $\sqrt{z}s_{\mu,\frac{1}{2}}(z)\to0$ as $z\to0,$ by using Lemma \ref{lopi} we obtain that the function
    $$z\mapsto\frac{\sqrt{z}s_{\mu+1,\frac{1}{2}}(z)}{\sqrt{z}s_{\mu,\frac{1}{2}}(z)}=\frac{s_{\mu+1,\frac{1}{2}}(z)}{s_{\mu,\frac{1}{2}}(z)}$$
    is also increasing on $(0,\xi_{\mu-1,1})$ for $\mu\in\left(-\frac{1}{2},\frac{1}{2}\right),$ and this is equivalent to the reversed form of the inequality \eqref{turan}. Moreover, changing again $\mu$ to $\mu-1$ we get that the function $z\mapsto s_{\mu,\frac{1}{2}}(z)/s_{\mu-1,\frac{1}{2}}(z)$ is increasing on $(0,\xi_{\mu-2,1})$  for $\mu\in\left(\frac{1}{2},\frac{3}{2}\right).$ But, this means that the function
    $$z\mapsto\frac{\left[\sqrt{z}s_{\mu+1,\frac{1}{2}}(z)\right]'}{\left[\sqrt{z}s_{\mu,\frac{1}{2}}(z)\right]'}=
    \frac{\left(\mu+\frac{1}{2}\right)\sqrt{z}s_{\mu,\frac{1}{2}}(z)}{\left(\mu-\frac{1}{2}\right)\sqrt{z}s_{\mu-1,\frac{1}{2}}(z)}=
    \frac{\mu+\frac{1}{2}}{\mu-\frac{1}{2}}\frac{s_{\mu,\frac{1}{2}}(z)}{s_{\mu-1,\frac{1}{2}}(z)}$$
    is increasing on $(0,\xi_{\mu-2,1})$ for $\mu\in\left(\frac{1}{2},\frac{3}{2}\right).$ Now, since $\sqrt{z}s_{\mu+1,\frac{1}{2}}(z)\to0$ and $\sqrt{z}s_{\mu,\frac{1}{2}}(z)\to0$ as $z\to0,$ by using again the monotone form of l'Hospital's rule (see Lemma \ref{lopi}) we obtain that the function
    $$z\mapsto\frac{\sqrt{z}s_{\mu+1,\frac{1}{2}}(z)}{\sqrt{z}s_{\mu,\frac{1}{2}}(z)}=\frac{s_{\mu+1,\frac{1}{2}}(z)}{s_{\mu,\frac{1}{2}}(z)}$$
    is also increasing on $(0,\xi_{\mu-2,1})$ for $\mu\in\left(\frac{1}{2},\frac{3}{2}\right),$ and this is equivalent to the fact that the Tur\'an type inequality \eqref{turan} is valid for $z\in(0,\xi_{\mu-2,1})$ and $\mu\in\left(\frac{1}{2},\frac{3}{2}\right).$ Combining this with part {\bf B} of these concluding remarks we get that the inequality \eqref{turan} is valid for all $z\in(0,\xi_{\mu-m,1})$ and $m-\frac{3}{2}<\mu<m-\frac{1}{2},$ $m\in\{2,3,\dots\}.$

{\bf D.} Finally, observe that the right-hand side of \eqref{turan} for $\mu>\frac{1}{2}$ is negative, and thus it is natural to ask whether the next Tur\'an type inequality (or its reverse) is valid for $\mu>\frac{1}{2}$ and $z>0$
    \begin{equation}\label{conj}\Delta_{\mu}(z):=\left[s_{\mu,\frac{1}{2}}(z)\right]^2-s_{\mu-1,\frac{1}{2}}(z)s_{\mu+1,\frac{1}{2}}(z)>0.\end{equation}
    However, for example the Tur\'an expression $\Delta_{\frac{3}{2}}(z)$ has an infinity of changes of sign on $(0,\infty).$ For this first observe that
    $$s_{\frac{1}{2},\frac{1}{2}}(z)=\frac{1-\cos z}{\sqrt{z}},\ s_{\frac{3}{2},\frac{1}{2}}(z)=\frac{z-\sin z}{\sqrt{z}}\ \ \ \mbox{and}
    \ \ \ s_{\frac{5}{2},\frac{1}{2}}(z)=\frac{z^2+2\cos z-2}{\sqrt{z}}.$$
    The first relation follows from the fact that when $\mu=\nu$ the Lommel function of the first kind can be expressed with the aid of the Struve function of order $\nu,$ for more details see \cite[p. 124]{stein}. The third relation follows from \cite[p. 348]{wats}
    $$s_{\mu+2,\nu}(z)=z^{\mu+1}-\left[(\mu+1)^2-\nu^2\right]s_{\mu,\nu}(z)$$
    by taking $\mu=\nu=\frac{1}{2},$ while the second relation follows from \eqref{diff} by choosing $\mu=\frac{5}{2}$ and $\nu=\frac{1}{2}.$
    Now, since $\eta(z):=z\Delta_{\frac{3}{2}}(z)=(z^2-4)\cos z+\cos^2z-2z\sin z+3$ takes the values
    $$\eta((2n-1)\pi)={8-(2n-1)^2\pi^2}<0\ \ \ \mbox{and}\ \ \
    \eta(2n\pi)=(2n\pi)^2>0$$
for $n\in\{1,2,\dots\},$ it is clear that indeed $\Delta_{\frac{3}{2}}(z)$ has an infinity of changes of sign on $(0,\infty).$

\end{document}